\documentclass[runningheads,a4paper]{llncs}

\usepackage{amssymb}
\usepackage{amsmath}
\usepackage{graphicx}
\usepackage{xcolor}
\usepackage{graphics}
  \usepackage{epsfig}
\usepackage{url}
\urldef{\mailsa}\path|{alfred.hofmann, ursula.barth, ingrid.haas, frank.holzwarth,|
\urldef{\mailsb}\path|anna.kramer, leonie.kunz, christine.reiss, nicole.sator,|
\urldef{\mailsc}\path|erika.siebert-cole, peter.strasser, lncs}@springer.com|
\newcommand{\keywords}[1]{\par\addvspace\baselineskip
\noindent\keywordname\enspace\ignorespaces#1}

\newcommand{\R}{\mathbb{R}}

\newcommand{\Il}{I^{\alpha}_{a+}}
\newcommand{\Ir}{I^{\alpha}_{b-}}
\newcommand{\Ipl}{I^{\alpha}_{0+}}
\newcommand{\Ipr}{I^{\alpha}_{T-}}

\newcommand{\DRl}{D^{\alpha}_{a+}}
\newcommand{\DRr}{D^{\alpha}_{b-}}
\newcommand{\Dpl}{D^{\alpha}_{0+}}
\newcommand{\Dpr}{D^{\alpha}_{T-}}
\newcommand{\Ila}{I^{1-\alpha}_{a+}}
\newcommand{\Ira}{I^{1-\alpha}_{b-}}
\newcommand{\Ipla}{I^{1-\alpha}_{0+}}

\newcommand{\al}{\alpha}

\begin{document}

\mainmatter
\title{Non-invasive control of the fractional Hegselmann--Krause type model.}

\titlerunning{Non-invasive control of the fractional Hegselmann--Krause type model.}

\author{Ricardo Almeida$^{1}$, Agnieszka B. Malinowska$^{2}$
\and Tatiana Odzijewicz$^{3}$
}
\authorrunning{R. Almeida, A. B. Malinowska, T. Odzijewicz}

\institute{$^{1}$Center for Research and Development in Mathematics and Applications (CIDMA)\\ Department of Mathematics,
University of Aveiro, Portugal\\
$^{2}$Faculty of Computer Science,
Bialystok University of Technology\\
15-351 Bia\l ystok, Poland\\
$^{3}$Department of Mathematics and Mathematical Economics\\
Warsaw School of Economics\\
02-554 Warsaw, Poland\\
}

\maketitle

\begin{abstract}
In this paper, the fractional order Hegselmann--Krause type model with leadership is studied. We seek an optimal control strategy for the system to reach a consensus in such a way that the control mechanism is included in the leader dynamics. Necessary optimality conditions are obtained by the use of a fractional counterpart of Pontryagin Maximum Principle. The effectiveness of the proposed control strategy is illustrated by numerical examples.
\end{abstract}

\section{Introduction}

In recent years consensus algorithms for multi-agent systems have been widely discussed in the literature, due to their potential
applications in biology \cite{Aoki,reynolds}, physics \cite{Toner,Vicsec} and engineering areas \cite{Jadbabaie,Kar,MPA2017,Olfati1}. The main idea of a consensus algorithm is to drive a team of agents to reach an agreement on a common goal (e.g. positions, velocity, opinion) by interacting with their neighbours. Consensus algorithms are based on nearest-neighbour rules \cite{Cucker_Smale_1,GMSZ2016,Jadbabaie,Olfati1}, bounded confidence \cite{Blondel2010,GMMM,Krause+Hegselmann2002,Mozyrska} or a virtual leader \cite{bai,ren,YJH2015}. The virtual leader is an agent whose motion is independent of all the other agents, and thus is followed by all the other ones. \\
If a consensus is not achieved, one can apply an optimal control strategy to the system to enforce convergence. For example, in \cite{Caponigro} a mathematical model of sparse control was designed in order to attain a consensus. A similar control strategy was proposed in \cite{MOS2017} to address a consensus problem in the fractional Cucker--Smale model.
In this paper, inspired by \cite{Wongkaew}, we propose a different approach. Namely, we introduce the virtual leader to the system and apply a control function to the leader. This control should steer asymptotically the system to consensus in the most economical way. Therefore, in the cost functional we minimize the transient state deviation and a control effort. Following \cite{Wongkaew}, we call this control strategy non-invasive.

The rest of the paper is organized as follows. In Section~\ref{prel}, we recall necessary concepts and facts on
fractional operators and fractional optimal control problems. For a deeper discussion of the fractional calculus and its applications we
refer the reader to \cite{almeida,Kaczorek,Kamocki1,Kamocki2,Kilbas,malin1,malin2,ostalczyk} and references therein. Main results are then
stated and proved in Section~\ref{sec:ocHK}, where we show the existence of optimal controls for the fractional Hegselmann--Krause type model with leadership and the necessary optimality conditions. In Section~\ref{examples}, simulation results are presented in order to demonstrate the validity of the proposed control strategy.


\section{Preliminaries}\label{prel}

In this section, we give notations and essential facts that will be used in the sequel.

Let $[a,b]\subset \R$ be any bounded interval. For $\al>0$ and $f\in L^1([a,b];\R^n)$ we define the left and the right Riemann--Liouville fractional integrals $\Il$ and $\Ir$ by
\begin{equation*}
\Il[f](t):=\frac{1}{\Gamma(\al)}\int\limits_a^t\frac{f(\tau)}{(t-\tau)^{1-\al}}\;d\tau, ~~t\in[a,b]~a.e.
\end{equation*}
\begin{equation*}
\Ir[f](t):=\frac{1}{\Gamma(\al)}\int\limits_t^b\frac{f(\tau)}{(\tau-t)^{1-\al}}\;d\tau, ~~t\in[a,b]~a.e.
\end{equation*}
Now, let us define
$$\Il(L^p([a,b];\R^n)):=\left\{f:[a,b]\rightarrow\R^n:\exists_{g\in L^p([a,b];\R^n)}f=\Il[g]\right\}$$
and
$$\Ir(L^p([a,b];\R^n)):=\left\{f:[a,b]\rightarrow\R^n:\exists_{g\in L^p([a,b];\R^n)}f=\Ir[g]\right\}.$$
For $\alpha \in (0,1)$ the left Riemann--Liouville fractional derivatives $\DRl$ are defined for functions $\Il f\in AC([a,b];\R^n)$ by
$$\DRl[f](t):=\frac{d}{dt}\Il[f](t),~~t\in[a,b]~a.e.$$
Similarly, for $\alpha \in (0,1)$ the right Riemann--Liouville fractional derivatives $\DRr$ are defined for functions $\Ir f\in AC([a,b];\R^n)$ by
$$\DRr[f](t):=\frac{d}{dt}\Ir[f](t),~~t\in[a,b]~a.e.$$
Consider the following fractional optimal control problem:

\begin{eqnarray}
\DRl[y](t)=g(t,y(t),u(t)),~~t\in [a,b]\textnormal{ a.e.},\label{eqn:FOC1:1}\\
\Ila [y](a)=y_0,\label{eqn:FOC1:2}\\
u(t)\in M\subset \R^m,~~t\in [a,b],\\
\mathcal{J}(y,u)=\int\limits_a^b f(t,y(t),u(t))\;dt\rightarrow\min,\label{eqn:FOC1:4}
\end{eqnarray}
where $f:[a,b]\times\R^n\times M\rightarrow\R$, $g:[a,b]\times\R^n\times M\rightarrow\R^n$ and $\alpha\in (0,1)$.

\begin{definition}
Suppose that
$$\mathcal{U}_M:=\left\{u\in L^1([a,b];\R^m):u(t)\in M,t\in[a,b]\right\}.$$
A pair $(y_*,u_*)\in \Il(L^p)\times \mathcal{U}_M$ is said to be locally optimal solution
to problem \eqref{eqn:FOC1:1}--\eqref{eqn:FOC1:4}, if $y_*$, corresponding to $u_*$,
solves \eqref{eqn:FOC1:1}--\eqref{eqn:FOC1:2} and there is a neighborhood $V$ of $y_*$ in $\Il(L^p)$
such that
$$\mathcal{J}(y_*,u_*)\leq\mathcal{J}(y,u)$$
for every pair $(y,u)\in V\times\mathcal{U}_M$ satisfying \eqref{eqn:FOC1:1}--\eqref{eqn:FOC1:2}.
\end{definition}

Let $\left\|\cdot\right\|$ denote the Euclidean norm. The following theorem is a fractional counterpart of Pontryagin Maximum Principle.

\begin{theorem}{[cf. Theorem~8 and Theorem~9, \cite{Kamocki2}]}\label{thm:Kam3}
Let $\al\in (0,1)$ and $1\leq p<\frac{1}{1-\al}$. We assume that $M$ is compact and the following assumptions are satisfied:
\begin{enumerate}
\item $g\in C^1$ with respect to $y\in\R^n$ and
\begin{enumerate}
\item $t \mapsto g(t,y,u)$ is measurable on $[a,b]$ for all $y\in\R^n$, $u\in M$, $u\mapsto g(t,y,u)$ is continuous on $M$ for $t\in [a,b]$ a.e. and all $y\in \R^n$;
\item there exists $L>0$ such that
$$\left\|g(t,y_1,u)-g(t,y_2,u)\right\|\leq L\left\|y_1-y_2\right\|$$
for $t\in [a,b]$ a.e. and all $y_1,y_2\in\R^n$, $u\in M$;
\item there exist $r\in L^p([a,b];\R)$ and $\gamma\geq 0$ such that
$$\left\|g(t,0,u)\right\|\leq r(t)+\gamma \left\|u\right\|$$
for $t\in [a,b]$ a.e. and all $u\in M$;
\end{enumerate}
\item $t\mapsto f(t,y,u)$ is measurable on $[a,b]$ for all $y\in\R^n$, $u\in M$ and $u\mapsto f(t,y,u)$ is continuous on $M$
for a.e. $t\in[a,b]$ and all $y\in\R^n$;
\item $f\in C^1$ with respect to $y\in\R^n$ and there exist $\bar{a}_1\in L^{1}([a,b],\R_0^{+})$,
$\bar{a}_2\in L^{p'}([a,b],\R_0^{+})$ $\left(\frac{1}{p}+\frac{1}{p'}=1\right)$, $\bar{C}_1,\bar{C}_2\geq 0$ such that
\begin{eqnarray}
\left\|f(t,y,u)\right\|\leq \bar{a}_1(t)+\bar{C}_1\left\|y\right\|^p,\\
\left\|\frac{\partial}{\partial y}f(t,y,u)\right\|\leq \bar{a}_2(t)+\bar{C}_2\left\|y\right\|^{p-1},
\end{eqnarray}
for a.e. $t\in[a,b]$ and all $y\in\R^n$, $u\in M$;
\item $t\mapsto \displaystyle\frac{\partial}{\partial y}g(t,y,u)$, $t\mapsto \displaystyle\frac{\partial}{\partial y}f(t,y,u)$ are measurable on $[a,b]$ for all
$y\in\R^n$, $u\in M$;
\item $u\mapsto \displaystyle\frac{\partial}{\partial y}g(t,y,u)$, $u\mapsto \displaystyle\frac{\partial}{\partial y}f(t,y,u)$ are continuous on $M$ for a.e. $t\in[a,b]$ and all $y\in\R^n$;
\item for a.e. $t\in [a,b]$ and all $y\in\R^n$ the set
\begin{equation}\label{set}
\tilde{Z}:=\left\{\left(f(t,y,u),g(t,y,u)\right)\in\R^{n+1},~u\in M\right\}
\end{equation}
is convex.
\end{enumerate}
If the pair
$$(y_*,u_*)\in \left(\Il(L^p)+\left\{\frac{d}{(t-a)^{1-\al}};d\in\R^n\right\}\right)\times \mathcal{U}_M$$
is a locally optimal solution to problem \eqref{eqn:FOC1:1}--\eqref{eqn:FOC1:4}, then there exists a function $\lambda\in \Ir (L^{p'})$,
such that
\begin{equation}
\DRr[\lambda](t)=\frac{\partial}{\partial y}g(t,y_*(t),u_*(t))^T\lambda(t)+\frac{\partial}{\partial y}f(t,y_*(t),u_*(t))
\end{equation}
for a.e. $t\in [a,b]$ and
\begin{equation}
\Ira[\lambda](b)=0.
\end{equation}
Moreover,
\begin{equation}
f(t,y_*(t),u_*(t))+\lambda(t)g(t,y_*(t),u_*(t))=\min\limits_{u\in M}\left\{f(t,y_*(t),u)+\lambda(t)g(t,y_*(t),u)\right\}
\end{equation}
for a.e. $t\in [a,b]$.
\end{theorem}

Let $g(t,y,u)=A(t)y(t)+B(t)u(t)$, where $A:[a,b]\rightarrow\R^{n\times n}$, $B:[a,b]\rightarrow\R^{n\times m}$. Then the following theorem, proved in \cite{Kamocki1}, ensures the existence of an optimal solution to problem \eqref{eqn:FOC1:1}--\eqref{eqn:FOC1:4}.

\begin{theorem}{[cf. Theorem 19, \cite{Kamocki1}]}\label{thm:Kam2}
Suppose that $1<p<\frac{1}{1-\al}$ and
\begin{enumerate}
\item $M$ is convex and compact;
\item $t\mapsto f(t,y,u)$ is measurable on $[a,b]$ for all $y\in\R^n$ and $u\in M$;
\item $(y,u)\mapsto f(t,y,u)$ is continuous on $\R^n\times M$ for a.e. $t\in [a,b]$;
\item $u\mapsto f(t,y,u)$ is convex on $M$ for a.e. $t\in [a,b]$ and all $y\in\R^n$;
\item $A,B$ are essentially bounded on $[a,b]$;
\item there exists a summable function $\psi_1:[a,b]\rightarrow\R_0^+$ and a constant $c_1\geq 0$ such that
\begin{equation*}
f(t,y,u)\geq -\psi_1(t)-c_1\left\|y\right\|
\end{equation*}
for a.e. $t\in[a,b]$ and all $y\in\R^n$, $u\in M$.
\end{enumerate}
Then problem \eqref{eqn:FOC1:1}--\eqref{eqn:FOC1:4} possesses an optimal solution
$$(y_0,u_0)\in\left(\Il(L^p)+\left\{\frac{d}{(t-a)^{1-\al}};d\in\R^n\right\}\right)\times \mathcal{U}_M.$$
\end{theorem}
\section{Optimal control of the fractional Hegselmann--Krause type model with leadership}\label{sec:ocHK}

In this section we investigate the fractional optimal control problem with the Hegselmann--Krause type dynamics and leadership. The agents' and the leader's opinions are denoted by $x_i$, $i=1,\ldots,N$, $N>1$, and $x_0$, respectively, and are represented by the state of the system ${\bf x}=(x_0,x_1,\dots,x_N)\in \R^{(N+1)d}$.

Let $\alpha\in (0,1)$ and $x:[0,T]\rightarrow \R^{(N+1)d}$. The control is an integrable function $t\mapsto u(t)\in\R^d$ such that
$$u(t)\in M:=\left\{u(t)\in\R^d:\left\|u\right\|_{l_2^d}\leq K\right\},$$
 where $\left\|\cdot\right\|_{l_2^d}$ denotes the Euclidean norm in $\R^d$. Consider the problem of finding trajectory solution to the system

\begin{equation}\label{eq:HKS1}
\begin{cases}
\Dpl [x_0](t)=u(t),\\
\Dpl [x_i](t)=\sum\limits_{j=1}^N a_{ij}(x_j(t)-x_i(t))+c_i(x_0(t)-x_i(t)),
\end{cases}
\end{equation}
$i=1,\dots,N$, initialized at $\Ipla [x_j](0)=x_{j0}\in\R^d$, $j=0,1,\dots, N$.
In system \eqref{eq:HKS1}, the term $\sum\limits_{j=1}^N a_{ij}(x_j(t)-x_i(t))$ comes from the classical Hegselmann--Krause model.
The weights $a_{ij}\geq0$ quantify the way that the agents influence each other. The second term $c_i(x_0(t)-x_i(t))$ describes the influence of the leader on the $i$th agent at the time $t$. In the case when the leader state is available to agent $i$ the value of $c_i$ is positive, otherwise $c_i=0$. System \eqref{eq:HKS1} can be written in the matrix form as follows:
$$\Dpl [x](t)=Ax(t)+Bu(t),$$
with $x=(x_0,x_1,\ldots,x_N)^T$, $u=(u_0,\ldots,\ldots,0)^T$,
\[A=   \left[\begin {array}{cccc} 0_{d}&0_{d}&\cdots&0_{d}\\ \noalign{\medskip}c_1I_{d}&s_1I_{d}&\cdots&a_{1N}I_{d}
\\ \noalign{\medskip}\cdots&\cdots&\ddots&\cdots\\ \noalign{\medskip}c_NI_{d}&a_{N1}I_{d}&\cdots&s_NI_{d}
\end {array} \right],
\,\quad B=\left[ \begin {array}{c} 1_{d}\\ \noalign{\medskip}0_{d}\\ \noalign{\medskip}\cdots\\ \noalign{\medskip}0_{d}
\end {array} \right]
\,,\]
where $s_i=-\bigl(\sum_{j\neq i}a_{ij}+c_i\bigr)$, $i=1,...,N$, $I_{d}$ is the identity and $0_{d}$ is the null matrix.

A solution to \eqref{eq:HKS1} has to minimize the following cost functional
\begin{equation}\label{eq:HKCF1}
\int\limits_0^T\Biggl[
\frac{1}{2N^2}\sum\limits_{i,j=1}^N\left\|x_i(t)-x_j(t)\right\|_{l_2^d}^2
+\frac{1}{2}\sum\limits_{i=1}^N\left\|x_0(t)-x_i(t)\right\|_{l_2^d}^2
+\frac{\nu}{2}\left\|u\right\|_{l_2^d}^2\Biggr]\;dt,
\end{equation}
where $\nu>0$ is a weight constant.

\subsection{Existence of solutions}

In this part of the text we discuss the question of existence of solutions to optimal control problem \eqref{eq:HKS1}--\eqref{eq:HKCF1}.

\begin{theorem}\label{thm:Ex:HK}
If $1< p<\frac{1}{1-\al}$, then the Hegselmann--Krause fractional optimal control problem \eqref{eq:HKS1}--\eqref{eq:HKCF1} has an optimal solution
$(x_*,u_*)$ in the set
$$\left(\Ipl\left(L^p([0,T];(\R^d)^{N+1}\right)+\left\{\displaystyle\frac{c}{t^{1-\al}};c\in(\R^d)^{N+1}\right\}\right)\times \mathcal{U}_M,$$
where
$$\mathcal{U}_M=\left\{u\in L^1\left([0,T];\R^d\right): u(t)\in M,~t\in [0,T]\right\}.$$
\end{theorem}

\begin{proof}
In order to obtain the desired result, it is enough to show that assumptions of Theorem~\ref{thm:Kam2} are satisfied. First, let us note that
set $M$ is compact (as it is closed and bounded) and let us define maps $g:[0,T]\times(\R^d)^{N+1}\times M\rightarrow(\R^d)^{N+1}$,
\begin{equation}\label{eq:FunctG}
g(t,x(t),u(t)):=\left[
\begin{array}{l}
u(t)\\
\sum\limits_{j=1}^N a_{ij}(x_j(t)-x_i(t))+c_i(x_0(t)-x_i(t))
\end{array}\right]_{i=1,\dots,N}
\end{equation}
and $f:[0,T]\times(\R^d)^{N+1}\times M\rightarrow\R$,
\begin{equation}\label{eq:FunctF}
f(t,x(t),u(t)):=\frac{1}{2N^2}\sum\limits_{i,j=1}^N\left\|x_i(t)-x_j(t)\right\|_{l_2^d}^2
+\frac{1}{2}\sum\limits_{i=1}^N\left\|x_0(t)-x_i(t)\right\|_{l_2^d}^2
+\frac{\nu}{2}\left\|u\right\|_{l_2^d}^2.
\end{equation}
It is easy to see that function $f$ is measurable w.r.t. $t$, continuous w.r.t. $(x,u)$ and convex w.r.t. $u$.
Moreover, because $A$ and $B$ are matrices with constant coefficients, they are essentially bounded on $[0,T]$.
Finally, let us note that
\begin{equation*}
f(t,x,u)=\frac{1}{2N^2}\sum\limits_{i,j=1}^N\left\|x_i(t)-x_j(t)\right\|_{l_2^d}^2
+\frac{1}{2}\sum\limits_{i=1}^N\left\|x_0(t)-x_i(t)\right\|_{l_2^d}^2
+\frac{\nu}{2}\left\|u\right\|_{l_2^d}^2\geq\frac{\nu}{2}\left\|u\right\|_{l_2^d}^2,
\end{equation*}
so choosing $\psi_1(t)=-\frac{\nu}{2}\left\|u\right\|_{l_2^d}^2$ and $c_1=0$, we see that the last assumption of Theorem~\ref{thm:Kam2} is satisfied.
\end{proof}

\subsection{Necessary optimality conditions}

In this section, in order to prove necessary optimality conditions for problem \eqref{eq:HKS1}--\eqref{eq:HKCF1}, we show that assumptions of  Theorem~\ref{thm:Kam3} are fulfilled. Let functions $f$ and $g$ be defined in the same way as in the proof of Theorem~\ref{thm:Ex:HK}. First, let us note that in \cite{Kamocki2} the author showed that under certain assumptions fractional optimal control problem \eqref{eqn:FOC1:1}--\eqref{eqn:FOC1:4} satisfies the Smooth Convex Extremum Principle (SCEP) (see, e.g., Theorem~3 of \cite{Kamocki2}) and as a consequence he obtained the Fractional Pontryagin Maximum Principle (Theorem~\ref{thm:Kam3}, see Theorems~8 and 9 of \cite{Kamocki2}).
Among other things, it was shown that if for a.e. $t\in [a,b]$ and all $y\in\R^n$ the set $\tilde{Z}$ (defined by \eqref{set})
is convex, then the convexity assumption in SCEP is satisfied. In the case of problem \eqref{eq:HKS1}--\eqref{eq:HKCF1} we have that
for all $u_1,u_2\in \mathcal{U}_M$, $x\in\Ipl\left(L^2([0,T];(\R^d)^{N+1})\right)+\left\{\displaystyle\frac{c}{t^{1-\al}};c\in(\R^d)^{N+1}\right\}$ and $\theta\in [0,1]$ the following hold:
$$\theta f(t,x,u_1)+(1-\theta)f(t,x,u_2)\geq f(t,x,\theta u_1+(1-\theta)u_2),$$
$$\theta g(t,x,u_1)+(1-\theta)g(t,x,u_2)= g(t,x,\theta u_1+(1-\theta)u_2)$$
with  $\theta u_1+(1-\theta)u_2\in \mathcal{U}_M$, by convexity of $M$. Therefore, the convexity assumption in SCEP is fulfilled.
One can easily check that $f$, $g$ are continuously differentiable with respect to $x\in (\R^d)^{N+1}$ for a.a. $t\in [0,T]$ and all $u\in\R^d$ and that $t\mapsto g(t,x(t),u(t))$, $t\mapsto f(t,x(t),u(t))$ are measurable on [0,T], also $u\mapsto f(t,x(t),u(t))$ is continuous on $M$. Moreover, $x\mapsto g(t,x,u)$ is Lipschitz, $u\mapsto g(t,x,u)$ is continuous and
$$\left\|g(t,0,u)\right\|_{l_1^{N+1}-l_2^d}=\left\|u\right\|_{l_1^{N+1}-l_2^d}.$$
Choosing $r(t)=0$ and $\gamma=1$  we conclude that assumption 1 of Theorem~\ref{thm:Kam3} is satisfied. Note that
\begin{equation}\label{eq:PochF}
\frac{\partial f}{\partial x}=\left[
\begin{array}{c}
Nx_0-\sum\limits_{i=1}^Nx_i\\
\frac{2}{N}(x_1-\frac{1}{N}\sum\limits_{j=1}^Nx_j)+(x_1-x_0)\\
\vdots\\
\frac{2}{N}(x_N-\frac{1}{N}\sum\limits_{j=1}^Nx_j)+(x_N-x_0)
\end{array}\right]
\end{equation}
and system \eqref{eq:HKS1} is linear. Therefore, $t\mapsto \frac{\partial f}{\partial x}(t,x,u)$,  $t\mapsto \frac{\partial g}{\partial x}(t,x,u)$ are measurable
on $[0,T]$ and $u\mapsto \frac{\partial f}{\partial x}(t,x,u)$,  $u\mapsto \frac{\partial g}{\partial x}(t,x,u)$
are continuous on $M$. Finally, let us check assumption 3 of Theorem~\ref{thm:Kam3}.
Observe that
\begin{multline*}
|f(t,x,u)|\leq
\frac{1}{2N^2}\sum\limits_{i,j=1}^N\left(\left\|x_i\right\|_{l_2^d}^2+2\left\|x_i\right\|_{l_2^d}\left\|x_j\right\|_{l_2^d}
+\left\|x_j\right\|_{l_2^d}^2\right)\\
+\frac{1}{2}\sum\limits_{i=1}^N\left(\left\|x_0\right\|_{l_2^d}^2+2\left\|x_0\right\|_{l_2^d}\left\|x_i\right\|_{l_2^d}
+\left\|x_i\right\|_{l_2^d}^2\right)+\frac{\nu}{2}\left\|u\right\|_{l_2^d}^2\\
=\left(\frac{1}{N}+\frac{1}{2}\right)\sum\limits_{i=1}^N\left\|x_i\right\|_{l_2^d}^2
+\left(\frac{1}{N^2}\sum\limits_{i=1}^N\left\|x_i\right\|_{l_2^d}+\left\|x_0\right\|_{l_2^d}\right)\sum\limits_{i=1}^N\left\|x_i\right\|_{l_2^d}
\\+\frac{N}{2}\left\|x_0\right\|_{l_2^d}^2
+\frac{\nu}{2}\left\|u\right\|_{l_2^d}^2
\leq\sum\limits_{i=0}^N\left\|x_i\right\|_{l_2^d}^2+\left(\sum\limits_{i=0}^N\left\|x_i\right\|_{l_2^d}\right)\left(\sum\limits_{i=1}^N\left\|x_i\right\|_{l_2^d}\right)
+\frac{\nu}{2}\left\|u\right\|_{l_2^d}^2\\
\leq\sum\limits_{i=1}^N\left\|x_i\right\|_{l_2^d}^2+\left(\sum\limits_{i=0}^N\left\|x_i\right\|_{l_2^d}\right)^2+\frac{\nu}{2}\left\|u\right\|_{l_2^d}^2\leq \frac{\nu}{2}\left\|u\right\|_{l_2^d}^2+2\left\|x\right\|_{l_1^{N+1}-l_2^d}^2,
\end{multline*}
and
\begin{multline*}
\left\|\frac{\partial f}{\partial x}(t,x,u)\right\|_{l_1^{N+1}-l_2^d}=\|Nx_0-\sum\limits_{i=1}^Nx_i\|_{l_2^d}
+\sum\limits_{i=1}^N\|\frac{2}{N}(x_i-\frac{1}{N}\sum\limits_{j=1}^Nx_j)+(x_i-x_0)\|_{l_2^d}\\
\leq N\|x_0\|_{l_2^d}+\sum\limits_{i=1}^N\|x_i\|_{l_2^d}+\left(\frac{2}{N}+1\right)\sum\limits_{i=1}^N \|x_i\|_{l_2^d}
+\frac{2}{N^2}\sum\limits_{i=1}^N\|\sum\limits_{j=1}^N x_j\|_{l_2^d}
+\sum\limits_{i=1}^N \|x_0\|_{l_2^d}\\
\leq (2+\frac{4}{N})\sum\limits_{i=0}^N\|x_i\|_{l_2^d}+2N\|x_0\|_{l_2^d}\leq
2N\left\|x\right\|_{l_1^{N+1}-l_2^d}.
\end{multline*}
Choosing $a_1(t)=\frac{\nu}{2}\left\|u\right\|_{l_2^d}^2$,
$a_2(t)= 0$ and $C_1=2$, $C_2=2N$ we get desired inequalities. Consequently, all assumptions of Theorem~\ref{thm:Kam3} are satisfied and we obtain the following theorem.

\begin{theorem}\label{NOC}
For $\alpha\in \left(\frac{1}{2},1\right)$, if
$$(x_*,u_*)\in \left(\Ipl\left(L^2([0,T];(\R^d)^{N+1})\right)+\left\{\displaystyle\frac{c}{t^{1-\al}};c\in(\R^d)^{N+1}\right\}\right)\times \mathcal{U}_M$$
is a locally optimal solution to problem \eqref{eq:HKS1}--\eqref{eq:HKCF1}, then there exists
a function\\
$\lambda\in\Ipr\left(L^2([0,T];(\R^d)^{N+1})\right)$ such that
\begin{equation*}
\Dpr[\lambda](t)=A^T\lambda(t)+\frac{\partial}{\partial x}f(t,x_*(t),u_*(t))
\end{equation*}
for a.e. $t\in[0,T]$ and
$$I^{1-\alpha}_{T-}[\lambda](T)=0,$$
where $\frac{\partial f}{\partial x}$ is given by \eqref{eq:PochF}. Furthermore,
\begin{multline}\label{eq:maxcon}
f(t,x_*(t),u_*(t))
+\lambda(t)g(t,x_*(t),u_*(t))\\
=\min\limits_{u\in M}\left\{f(t,x_*(t),u)
+\lambda(t)g(t,x_*(t),u)\right\}
\end{multline}
for a.e. $t\in [0,T]$, where $f$ and $g$ are given by \eqref{eq:FunctF} and \eqref{eq:FunctG}, respectively.
\end{theorem}
\section{Illustrative examples}\label{examples}

In this section, two numerical examples are given to verify the effectiveness of the proposed control strategy.

\begin{example}
Let us consider the following system
\begin{equation}\label{eq:HKS:Ex1}
\begin{cases}
\Dpl [x_0](t)=u(t),\\
\Dpl [x_1](t)=x_2(t)-x_1(t)+x_0(t)-x_1(t),\\
\Dpl [x_2](t)=x_1(t)-x_2(t),\\
\Dpl [x_3](t)=x_4(t)-x_3(t)+x_0(t)-x_3(t),\\
\Dpl [x_4](t)=x_3(t)-x_4(t),\\
\Ipla [x_j](0)=x_{j0}\in\R,~j=0,1,2,3,4
\end{cases}
\end{equation}
where $u(t)\in M:=\left\{u(t)\in\R:|u(t)|\leq 1\right\}$. In this case, all agents are connected with each other through the leader $x_0$ (see Figure~\ref{example_conf1}). Without the presence of the leader, the set $\{x_1,x_2\}$ does not interact with the set $\{x_3,x_4\}$ and the dynamic is described by the system
\begin{equation}\label{without_control}
\begin{cases}
\Dpl [x_1](t)=x_2(t)-x_1(t),\\
\Dpl [x_2](t)=x_1(t)-x_2(t),\\
\Dpl [x_3](t)=x_4(t)-x_3(t),\\
\Dpl [x_4](t)=x_3(t)-x_4(t)),\\
\Ipla [x_j](0)=x_{j0}\in\R,~j=1,2,3,4.
\end{cases}
\end{equation}

\begin{figure}[h]\label{example_conf1}
  \centering
  \includegraphics[width=7cm]{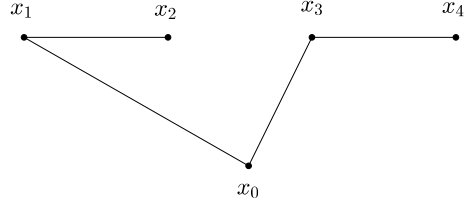}\\
  \caption{Model with leader and control.}
\end{figure}


Let us consider the functional
\begin{equation}\label{eq:HK:Ex2}
\int\limits_0^T\Biggl[
\frac{1}{32}\sum\limits_{i,j=1}^4(x_i(t)-x_j(t))^2+\frac{1}{2}\sum\limits_{i=1}^4(x_0(t)-x_i(t))^2+u^2(t)\Biggr]\;dt.
\end{equation}
One can easily check that problem of minimizing \eqref{eq:HK:Ex2} subject to \eqref{eq:HKS:Ex1} satisfies assumptions of Theorem~\ref{NOC}.
Let $(x_*,u_*)$ be a solution to problem \eqref{eq:HKS:Ex1}--\eqref{eq:HK:Ex2}, then there exists a function
$\lambda\in\Ipr\left(L^2([0,T];\R^5)\right)$ such that the triple $(x_*,u_*,\lambda)$ satisfies the system
\begin{equation}\label{NOC2}
\Dpr[\lambda](t)=A^T\lambda(t)+\left[
\begin{array}{cccc}
4x_{*0}-\sum\limits_{i=1}^4x_{*i} \\
\frac{1}{2}(x_{*1}-\frac{1}{4}\sum\limits_{i=1}^4x_{*i})+(x_{*1}-x_{*0}) \\
\frac{1}{2}(x_{*2}-\frac{1}{4}\sum\limits_{i=1}^4x_{*i})+(x_{*2}-x_{*0}) \\
\frac{1}{2}(x_{*3}-\frac{1}{4}\sum\limits_{i=1}^4x_{*i})+(x_{*3}-x_{*0})\\
\frac{1}{2}(x_{*4}-\frac{1}{4}\sum\limits_{i=1}^4x_{*i})+(x_{*4}-x_{*0})
\end{array}\right]
\end{equation}
with
$$I^{1-\alpha}_{T-}[\lambda](T)=0.$$
Furthermore,
$$F(x_*(t),u_*(t),\lambda(t))=\min\limits_{u\in M}F(x_*(t),u,\lambda(t)),~\textnormal{for a.e. }t\in [0,T],$$
where
\begin{multline*}
F(v,u,\lambda)
:=\frac{1}{32}\sum\limits_{i,j=1}^4(x_i-x_j)^2+\frac{1}{2}\sum\limits_{i=1}^4(x_0-x_i)^2+u^2+\lambda_0 u\\
+\lambda_1(x_0-2x_1+x_2)+\lambda_2(x_1-x_2)+\lambda_3(x_0-2x_3+x_4)+\lambda_4(x_3-x_4).
\end{multline*}
Note that variables $x$ do not influence on the point, where the minimum of $F$ is attained, but only on its value.
Therefore, an optimal control $u_*$ must be such that
\begin{equation*}
u_*^2(t)+\lambda_0(t)u(t)
=\min\limits_{u\in M}\{u^2(t)+\lambda_0(t) u(t)\}
\end{equation*}
for a.e. $t\in [0,T]$. Hence
\begin{equation*}
u_*(t)=
\begin{cases}
1& \text{ if } \lambda_0(t)\leq -2,\\
-\frac{\lambda_0(t)}{2} & \text{ if } -2 <\lambda_0(t)< 2,\\
-1 & \text{ if } \lambda_0(t)\geq 2,
\end{cases}
\end{equation*}
where $u_*$, $\lambda$ satisfy equations \eqref{eq:HKS:Ex1}, \eqref{NOC2}, and condition
$I^{1-\alpha}_{T-}[\lambda](T)=0$ holds.

Figures below show solutions to systems with and without leader and control, for the fractional orders $\alpha=0.6$ (Figure \ref{fig_Ex1alpha0.6}) and $\alpha=0.9$ (Figure \ref{fig_Ex1alpha0.9}). We see that, with the presence of the leader and control, the system converges to a consensus faster. The effectiveness of the control strategy is verified.

\begin{figure}[h]
\centering
\begin{minipage}{.5\textwidth}
  \centering
 \hspace*{-.9cm} \includegraphics[scale=0.3]{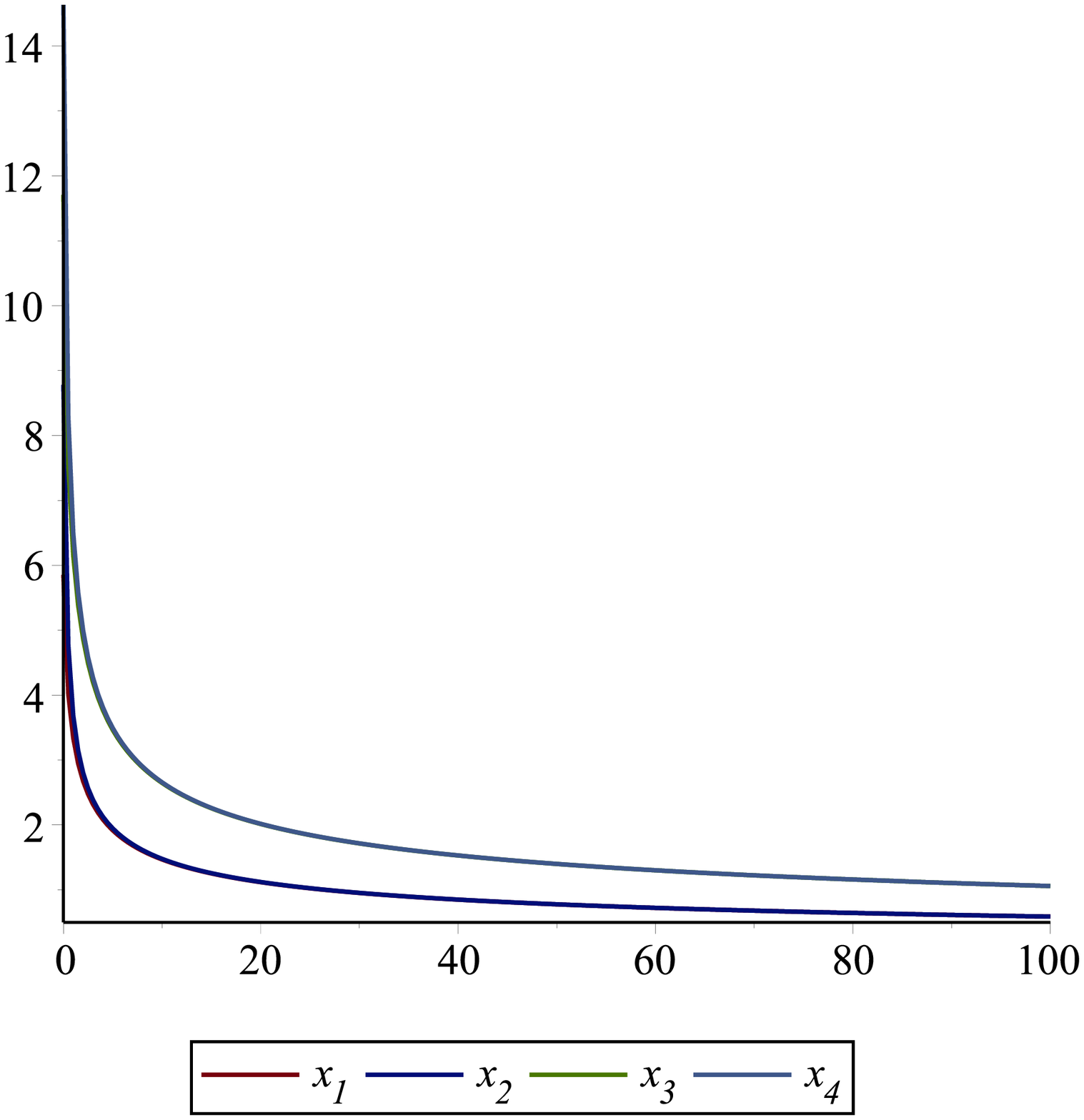}
\end{minipage}%
\begin{minipage}{.6\textwidth}
  \centering
  \hspace*{-1cm}  \includegraphics[scale=0.3]{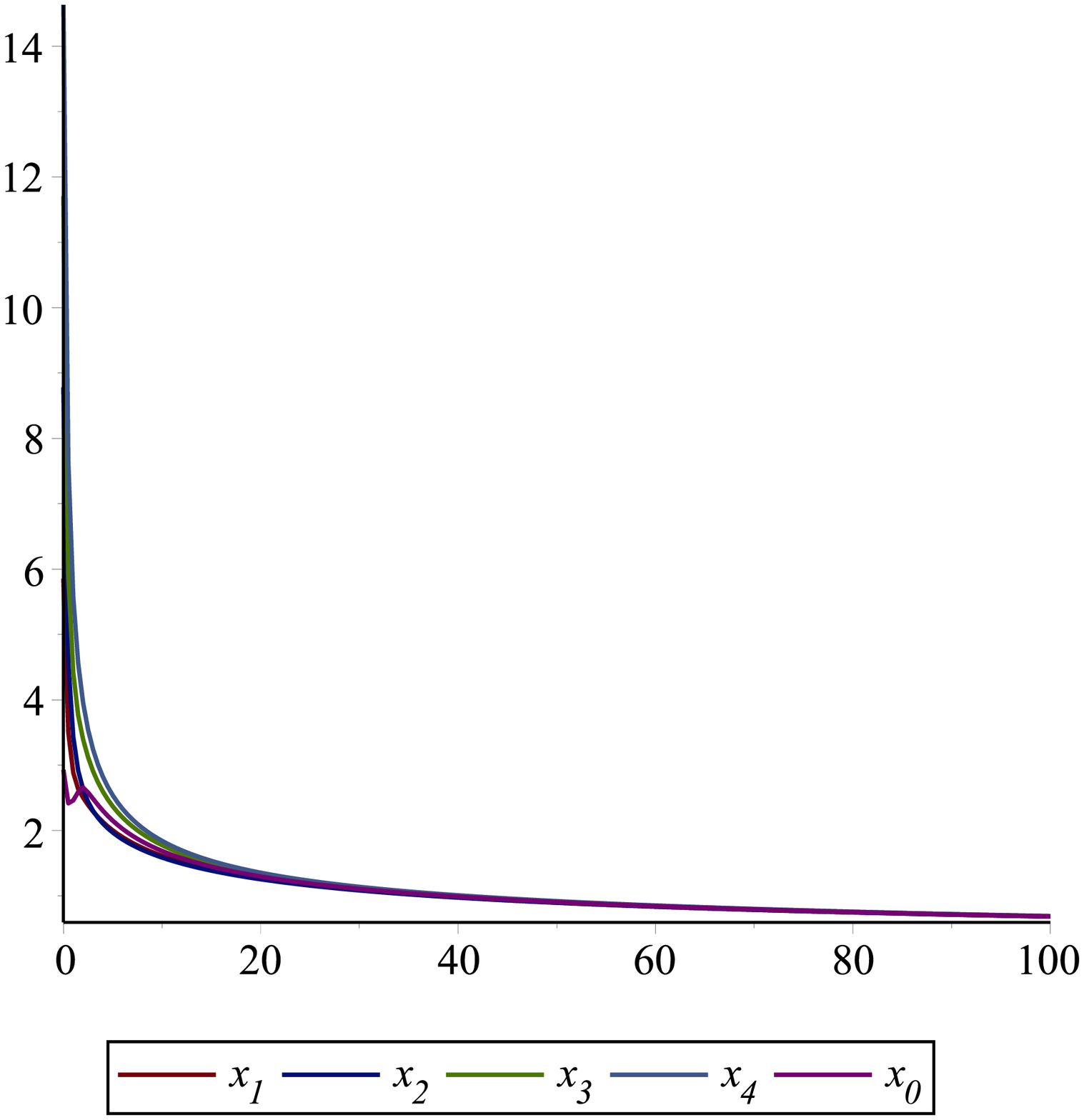}
\end{minipage}
\caption{Without  (left) and with  (right)  leader and control, for $\alpha=0.6$.}\label{fig_Ex1alpha0.6}
\end{figure}

\begin{figure}[h]
\centering
\begin{minipage}{.5\textwidth}
  \centering
 \hspace*{-.9cm} \includegraphics[scale=0.3]{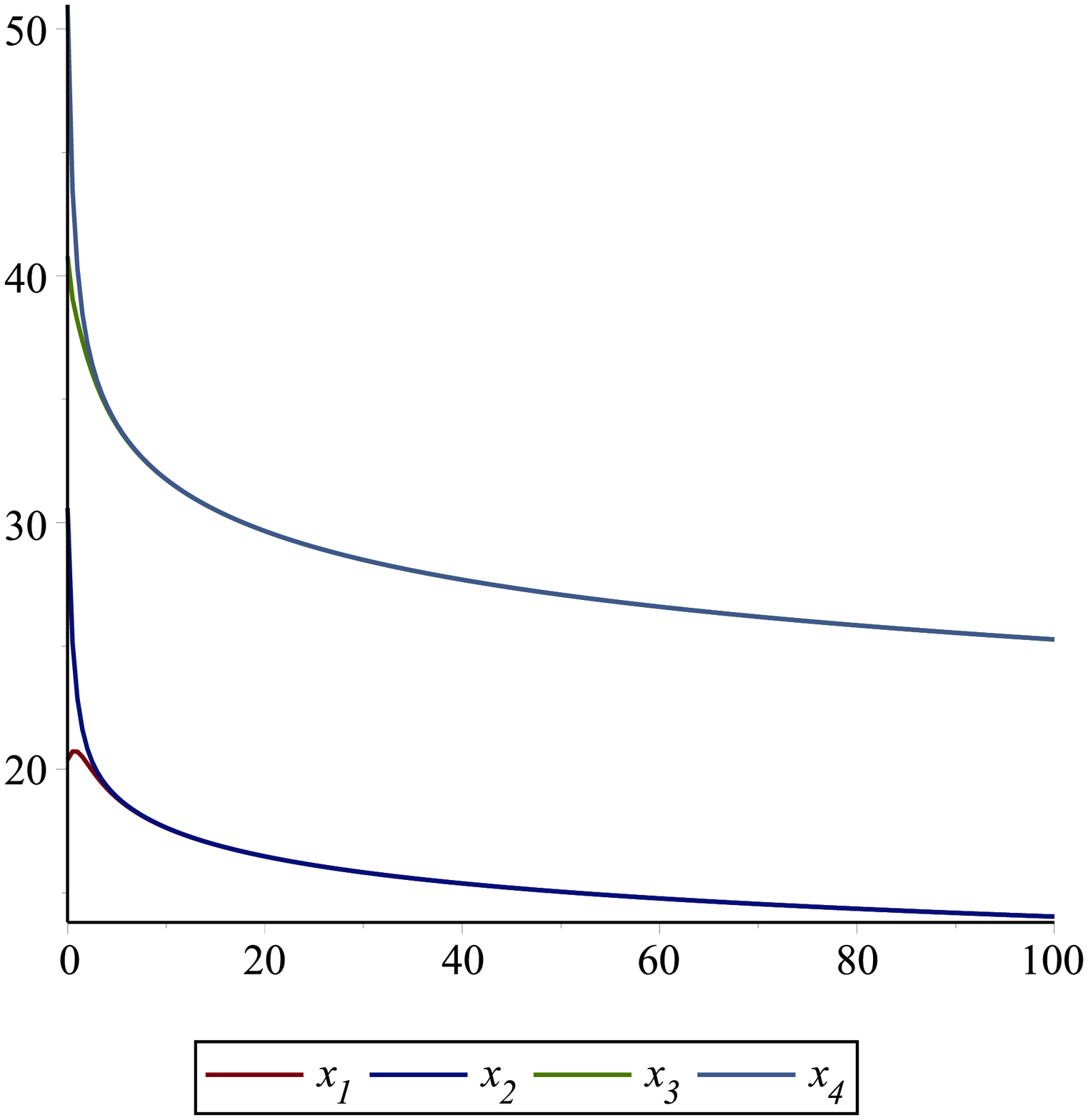}
\end{minipage}%
\begin{minipage}{.6\textwidth}
  \centering
  \hspace*{-1cm}  \includegraphics[scale=0.3]{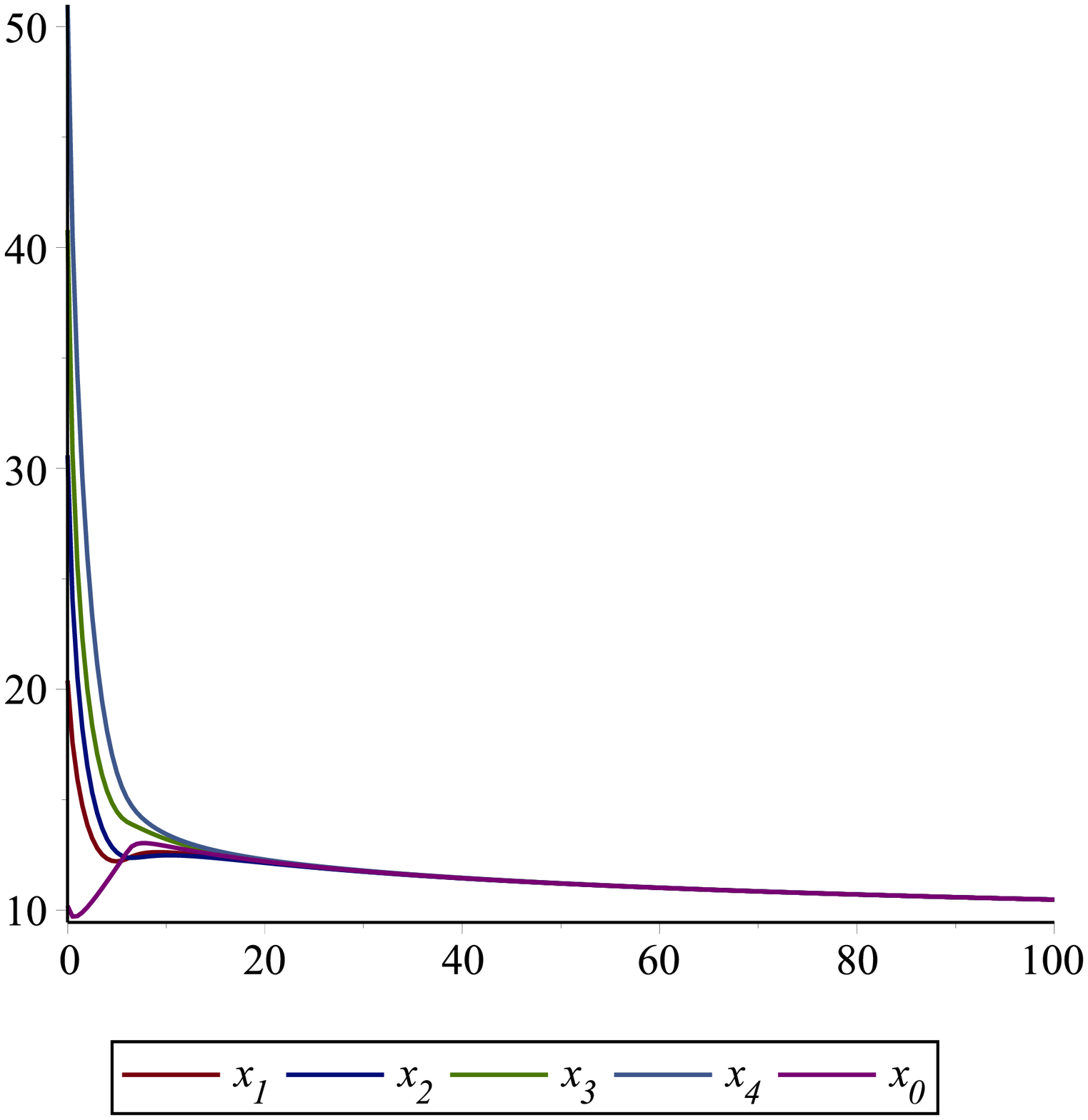}
\end{minipage}
\caption{Without  (left) and with  (right)  leader and control, for $\alpha=0.9$.}\label{fig_Ex1alpha0.9}
\end{figure}

\end{example}


\begin{example}
For our second example, consider a system given by three agents and the leader:
\begin{equation}\label{eq:HKS:Ex2}
\begin{cases}
\Dpl [x_0](t)=u(t),\\
\Dpl [x_1](t)=x_2(t)-x_1(t),\\
\Dpl [x_2](t)=x_1(t)-x_2(t),\\
\Dpl [x_3](t)=x_0(t)-x_3(t),\\
\Ipla [x_j](0)=x_{j0}\in\R,~j=0,1,2,3
\end{cases}
\end{equation}
where $u(t)\in M:=\left\{u(t)\in\R:|u(t)|\leq 10\right\}$ (Figure \ref{example_conf2}). In this model, agent $x_3$ does not interact with agents $x_1$ and $x_2$, independently of the leader presence.
\begin{figure}[h]
  \centering
  \includegraphics[width=7cm]{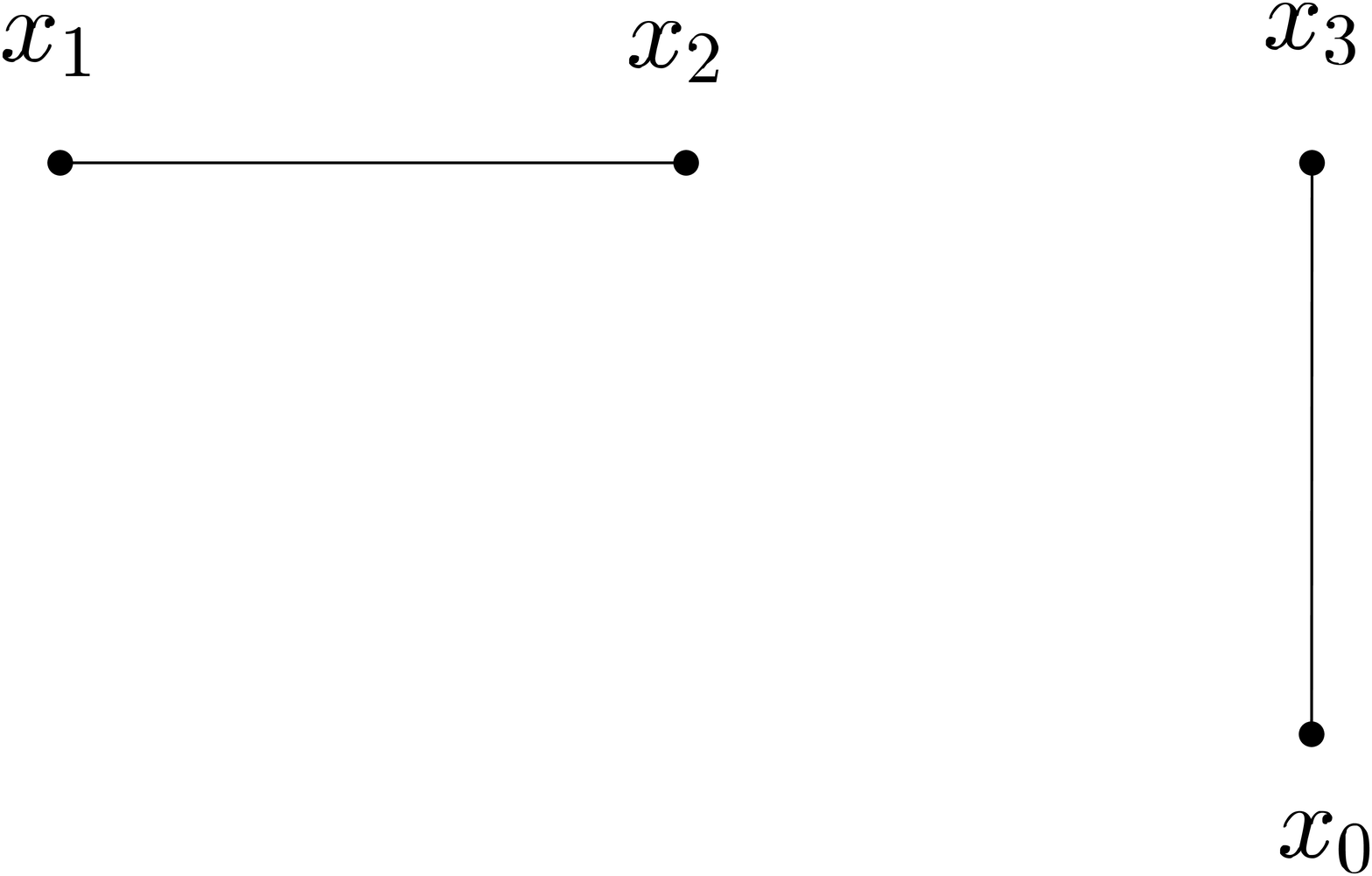}\\
  \caption{Model with leader and control.}\label{example_conf2}
\end{figure}

The objective is to minimize the functional
$$\int\limits_0^T\Biggl[\frac{1}{18}\sum\limits_{i,j=1}^3(x_i(t)-x_j(t))^2+\frac{1}{2}\sum\limits_{i=1}^3(x_0(t)-x_i(t))^2+u^2(t)\Biggr]\;dt,$$
subject to system \eqref{eq:HKS:Ex2}.
Considering the augmented function
\begin{multline*}
F(v,u,\lambda)
:=\frac{1}{18}\sum\limits_{i,j=1}^3(x_i-x_j)^2+\frac{1}{2}\sum\limits_{i=1}^3(x_0-x_i)^2+u^2+\lambda_0 u\\
+\lambda_1(-x_1+x_2)+\lambda_2(x_1-x_2)+\lambda_3(x_0-x_3),
\end{multline*}
we deduce that the optimal control $u_*$ is given by the formula
\begin{equation*}
u_*(t)=
\begin{cases}
1 & \text{ if } \lambda_0(t)\leq -20,\\
-\frac{\lambda_0(t)}{2} & \text{ if } -20 <\lambda_0(t)< 20,\\
-1 & \text{ if } \lambda_0(t)\geq 20.
\end{cases}
\end{equation*}
In figures below we present plots of agents' trajectories, with respect to the orders $\alpha=0.6$ (Figure \ref{fig_Ex2alpha0.6}) and $\alpha=0.9$ (Figure \ref{fig_Ex2alpha0.9}). Again, because of the presence of the leader and control a consensus is reached faster. This demonstrates the effectiveness of the applied control strategy.
\begin{figure}[h]
\centering
\begin{minipage}{.5\textwidth}
  \centering
 \hspace*{-.9cm} \includegraphics[scale=0.3]{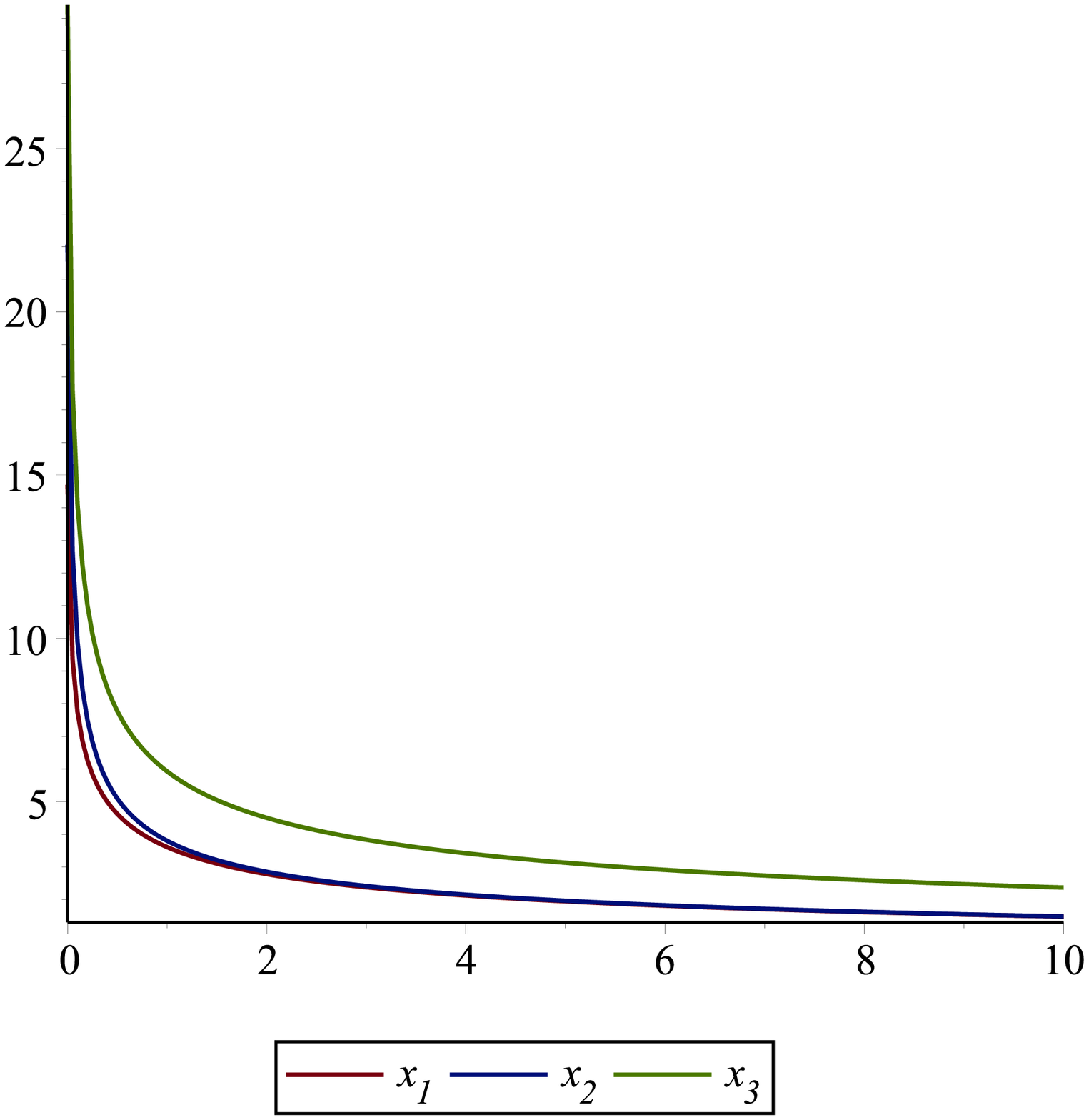}
\end{minipage}%
\begin{minipage}{.6\textwidth}
  \centering
  \hspace*{-1cm}  \includegraphics[scale=0.3]{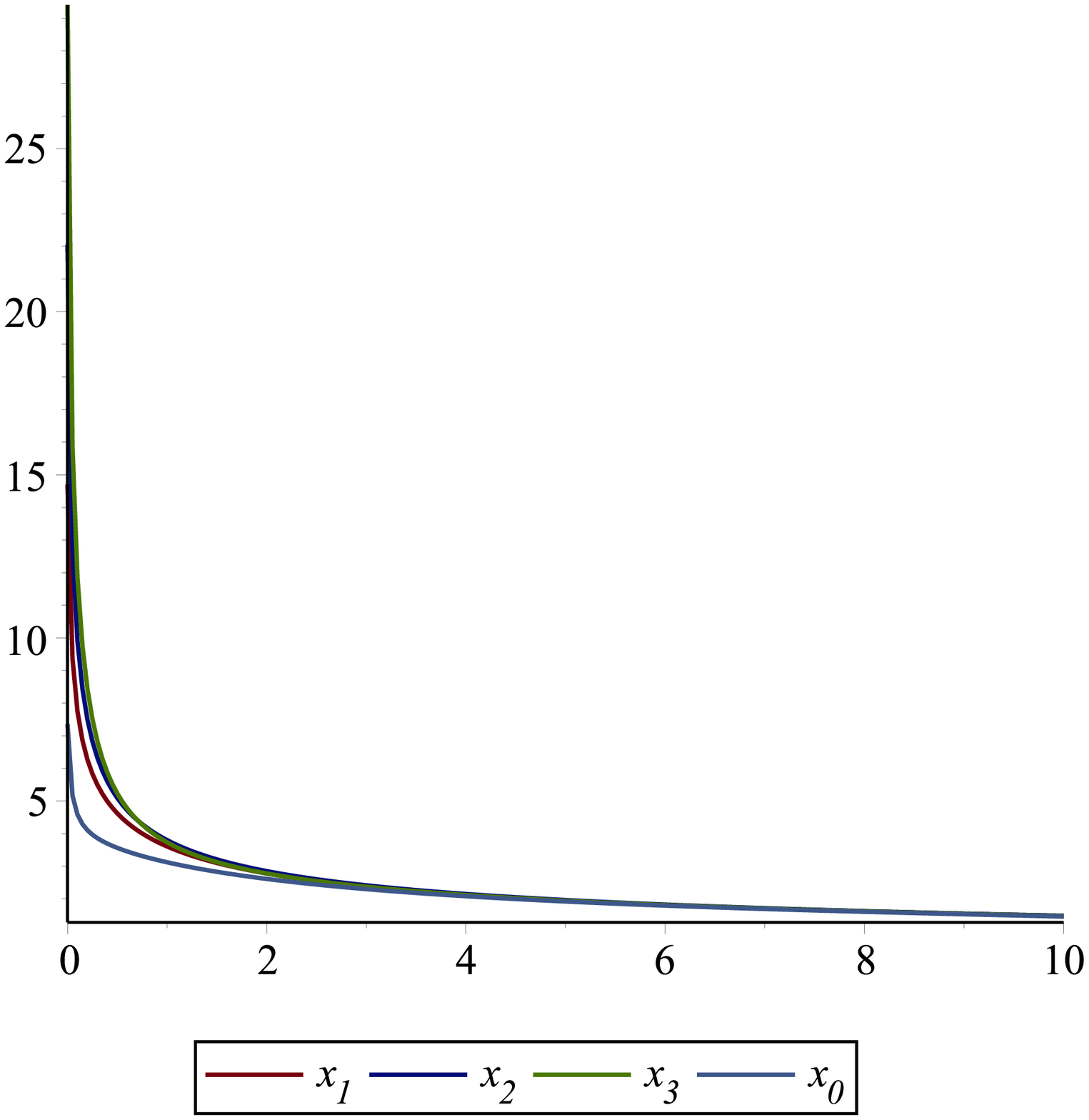}
\end{minipage}
\caption{Without  (left) and with  (right)  leader and control, for $\alpha=0.6$.}\label{fig_Ex2alpha0.6}
\end{figure}

\begin{figure}[h]
\centering
\begin{minipage}{.5\textwidth}
  \centering
 \hspace*{-.9cm} \includegraphics[scale=0.3]{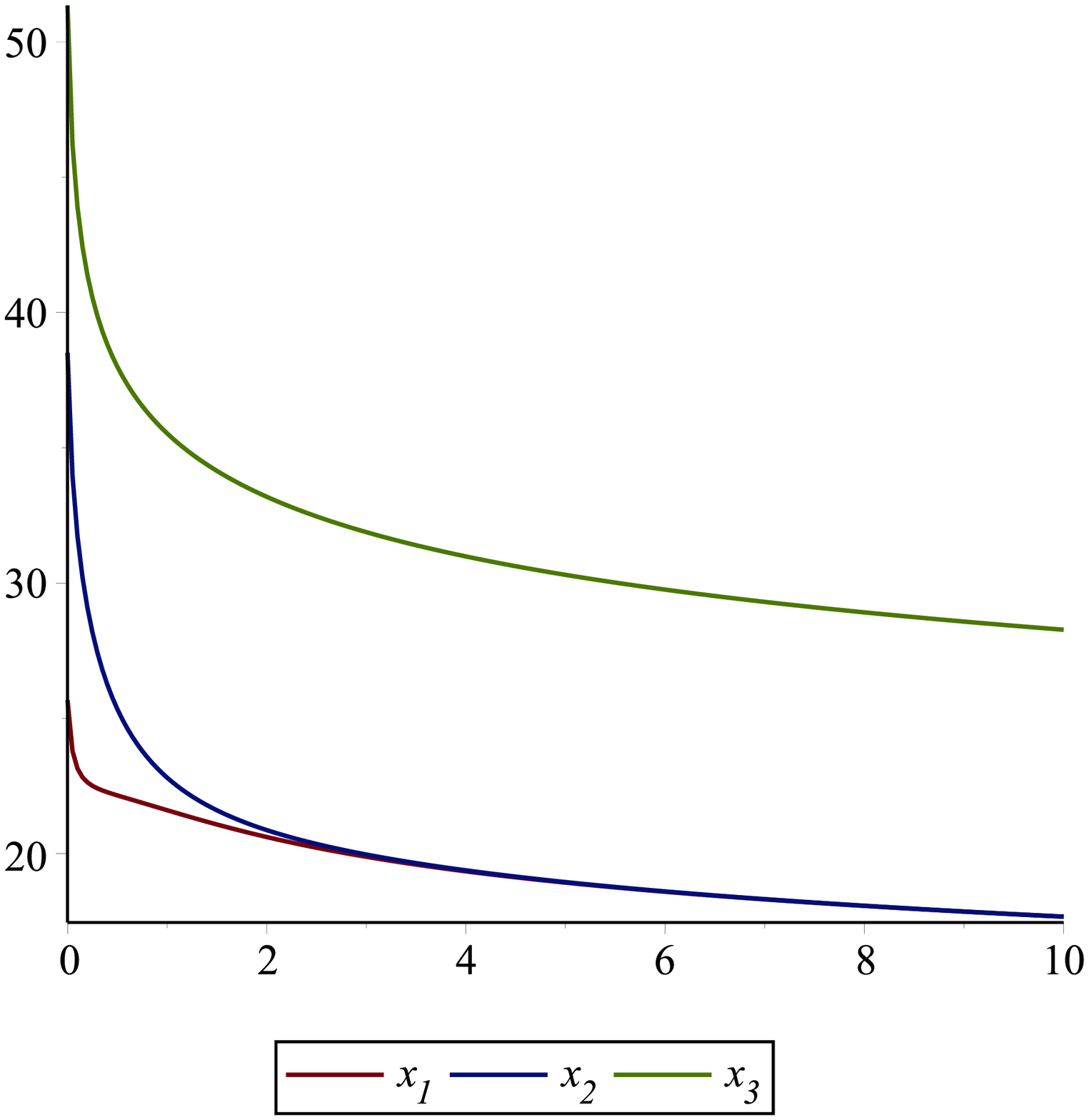}
\end{minipage}%
\begin{minipage}{.6\textwidth}
  \centering
  \hspace*{-1cm}  \includegraphics[scale=0.3]{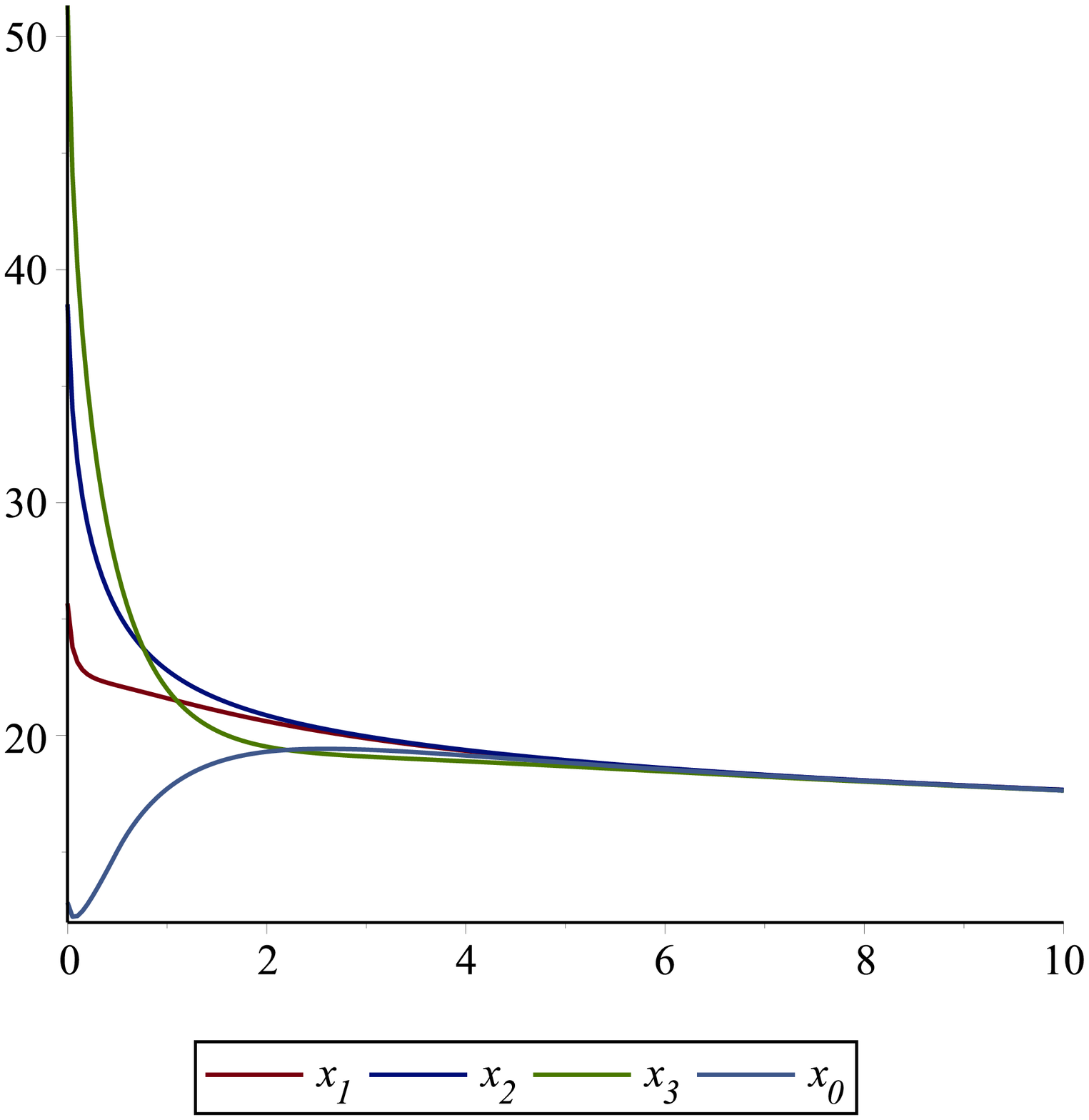}
\end{minipage}
\caption{Without  (left) and with  (right)  leader and control, for $\alpha=0.9$.}\label{fig_Ex2alpha0.9}
\end{figure}

\end{example}


\section*{Acknowledgements}
R. Almeida was supported by Portuguese funds through the CIDMA - Center for Research and Development in Mathematics and Applications, and the Portuguese Foundation for Science and Technology (FCT-Funda\c{c}\~ao para a Ci\^encia e a Tecnologia), within project UID/MAT/04106/2013; A. B. Malinowska and T. Odzijewicz were supported by Polish founds of National Science Center, granted on the basis of decision DEC-2014/15/B/ST7/05270.

\end{document}